\documentclass[11 pt]{amsart}
\usepackage{graphicx}
\usepackage{epstopdf}
\usepackage[all]{xy}
\usepackage[portuguese, english]{babel}
\usepackage{amsmath}
\usepackage{amssymb}
\usepackage{cite}
\usepackage{url}
\PassOptionsToPackage{hyphens}{url}\usepackage{hyperref}
\usepackage{color}
\usepackage[utf8]{inputenc}
\usepackage{hhline}
\usepackage{mathtools}

\newtheorem{theorem}{Theorem}[section]
\newtheorem{lemma}[theorem]{Lemma}
\newtheorem{prop}[theorem]{Proposition}

\theoremstyle{definition}

\numberwithin{equation}{section}

\usepackage{etoolbox}

\allowdisplaybreaks

\makeatletter

\def    \R      {{\mathbb R}}

\begin{document}

\title[Finiteness for Cross Central Configurations]{ Generic Finiteness for a Class of Symmetric Planar Central Configurations of the Six-Body Problem and the Six-Vortex Problem}

\author[Dias]{Thiago Dias}
\address{Department of Mathematics, Universidade Federal Rural de Pernambuco, av. Don Manoel de Medeiros s/n, Dois Irmãos - Recife - PE
52171-900, Brasil}
\curraddr{Department of Mathematics, National Tsing Hua University, Hsinchu City, 30013, Taiwan }
\email{thiago.diasoliveira@ufrpe.br}

\author[Pan]{Bo-Yu Pan}
\address{Department of Mathematics, National Tsing Hua University, Hsinchu City, 30013, Taiwan }
\email{bypan@math.nthu.edu.tw }

\date{}
\begin{abstract}
A symmetric planar central configuration of the Newtonian six-body problem $x$ is called \emph{cross central configuration} if there are precisely four bodies on a symmetry line of $x$.  We use complex algebraic geometry and Groebner basis theory to prove that for a generic choice of positive real masses $m_1,m_2,m_3,m_4,m_5=m_6$ there is a finite number of cross central configurations. We also show one explicit example of a configuration in this class. A part of our approach is based on relaxing the output of the Groebner basis computations. This procedure allows us to obtain upper bounds for the dimension of an algebraic variety.  We get the same results considering cross central configurations of the six-vortex problem. 
\end{abstract}
\maketitle

\section{Introduction}

 One of the leading open questions in the central configurations theory is the finiteness problem:  For every choice of $n$ point mass  $m_1,...,m_n$, is the number of central configurations finite?

  Chazy and Wintner contributed significantly to the interest in this problem that appears in the Smale's list for the Mathematicians of the 21st century \cite{SM}. Hampton and Moeckel used BKK theory to obtain the finiteness for central configurations of the four-body problem in the Newtonian case \cite{HM1} and the vortex case~\cite{HM2}.  Albouy and Kaloshin proved that for a choice of masses $m_1,...,m_5$ in the complement of a codimension-$2$ algebraic variety on the mass space, there is a finite number of planar central configurations of the Newtonian five-body problem  \cite{AK}. They studied the behavior of unbounded singular sequences of normalized central configurations  going to the infinity. 

In this paper, we consider symmetric planar central configurations with four points on a symmetry line in the context of the Newtonian six-body problem and the six-vortex problem. This type of configuration will be called {cross central configuration}.  
In the last years,  symmetric central configurations received much attention.  Leandro proved finiteness and studied bifurcations for a class of $d$-dimensional symmetric central configurations with $d+2$ bodies by using the method of rational parametrization \cite{LE}.  Albouy proved that every central configuration of the four-body problem with four equal masses is symmetric \cite{AA}.  Albouy, Fu, and Su provided the necessary and sufficient condition for a  planar convex four-body central configuration be symmetric with respect to one of its diagonals \cite{AFS}. Problems involving existence or enumeration of symmetric central configurations  satisfying some geometrical constraints were considered for many researchers (See for instance \cite{FL}, \cite{CZ}, \cite{GL}, \cite{JS}, \cite{PST}, and \cite{XZ}). Montaldi proved that there is a central configuration for every choice of a symmetry type and symmetric choice of mass \cite{MO}.
 
 The complex algebraic geometry has been used in the last decades  to study central configurations. We mention some papers related to this work. O'Neil utilized results on regular maps to count the number of equilibria and rigid translation configurations in the $n$-vortex problem \cite{ON}. Hampton, Roberts, and Santoprete studied relative equilibria of the four-vortex problem with two pairs of equal vorticities \cite{HRS}. They used exciting ideas involving Groebner basis computations, elimination theory, and the Jacobian criterion. Tsai applied the Hermite root counting theory and Groebner basis to obtain an exact counting theorem for special cases of the four-vortex problem \cite{TS}.
Moeckel proved the generic finiteness for Dziobek configurations of the Newtonian four-body problem \cite{RM1}. The fundamental tools used in this work to obtain the finiteness  were the Sard theorem for complex algebraic varieties and results due to Whitney about the structure of real algebraic varieties \cite{WT}. Moeckel utilized resultants and the dimension of fibers theorem to prove the generic finiteness for Dziobek configurations of the Newtonian $n$-body problem  \cite{RM2}, and  Dias used the Jacobian criterion to generalize this last result of Moeckel for potentials with semi-integer exponents \cite{DI}. The main result of the present paper is:

\begin{theorem}\label{main}
There is a proper open set $B$ of  $\mathbb{R}^5$ such that if $(m_i)=(m_1,...,m_5)\in \mathbb{R}^5\setminus B$ the number of cross central configurations of the Newtonian six-body problem is finite. 
\end{theorem}

In section $\ref{AVF}$, we describe cross central configurations with a polynomial system based on Laura-Andoyer equations with $32$ variables and $28$ equations that defines an algebraic variety $\Omega$. The mass space has dimension five. Hence, by the dimension of fibers theorem, in order to obtain our finiteness result, it is sufficient to show that $\text{dim}(\Omega)\leq 5$. To do this, we study the fibers of the projection of $\Omega$ on the mutual distances space by applying the Jacobian criterion and the dimension of fibers theorem.  The advantage of using the Jacobian criterion lies in the fact that this method reduces the problem of determining the dimension of an algebraic variety to the computation of the rank of the Jacobian matrix. In this way, if the Jacobian matrix is sufficiently sparse, this method can be effective even when the number of variables of the polynomial systems studied is huge. In other words, the Jacobian criterion is useful to compute the dimension of algebraic varieties defined by polynomials with many variables and few terms.  In our case, the Jacobian matrix presented in section $\ref{FN}$ is block-triangular. This fact allows us to reduce our study to determine the rank of the $4\times 5$ matrix with polynomial entries.  In our proof, we also use Groebner basis theory to do the rank computations. To determine a Groebner basis for the central configuration polynomials systems is a tough computational task in general. To deal with this problem, in the lemma \ref{lemmaD3} we compute a partial Groebner basis with a sufficient number of leading terms to obtain an upper bound for the dimension of certain irreducible components of the algebraic variety $\Omega$. In section $\ref{anexample}$, we use basic elimination theory to obtain an example of a cross central configuration that play an important role in our argument.  In section \ref{Vortexcase}, we prove a version of Theorem \ref{main} for planar central configurations of the six-vortex problem. At this point, we obtain the first finiteness results for the planar six-body problem even restricting to particular classes. In sections $\ref{anexample}$ and $\ref{ARG}$, we include results from algebraic geometry making our exposition self-contained. The sections \ref{AVF} and \ref{FN} contains computer-aided proofs. We make the computations on SageMath \cite{sage} and Singular \cite{DGPS}.  The SageMath notebooks can be found in \cite{DIG}.

\section{Cross Central Configurations}\label{formul}
In this section, we provide appropriate polynomial parametrization for the mutual distances associated with a cross central configuration.

A configuration $x=(x_1,...,x_n)\in \mathbb{R}^{dn}\setminus\Delta$,   is called \emph{central} in $\mathbb{R}^{d}$ if there exists $\lambda \neq 0$ such that
\begin{equation}\label{eqcc}
\sum_{j\neq i}m_{j}(x_{j}-x_{i})r_{ij}^{-3}+\lambda(x_i-c)=0, \qquad i=1,...,n,
\end{equation}
where, 
$$\Delta \;=\;\{x=(x_1,...,x_n) \in (\R^d)^n: x_i=x_j\; \text{for some}\; i\neq j\},\quad r_{ij}=\|x_i-x_j\|,$$
$$ c=\frac{1}{M}\left(m_1x_1+\cdots m_nx_n\right) \quad \text{and} \quad M=m_1+\cdots m_n \neq 0$$
are, respectively, the \emph{collision configuration set}, \emph{the mutual distances of the bodies}, the  \emph{center of mass} and the \emph{total mass}.

The dimension of a configuration $x$, denoted by $\delta(x)$, is the dimension of the smallest affine space that contains the points  $x_1,...,x_n \in \mathbb{R}^{d}.$ In this work, we only consider central configurations with six bodies and dimension two. An excellent introductory text about central configurations theory is \cite[Ch.II]{LMS}.

Central configurations are the initial conditions for homographic orbits of the $n$-body problem. They are invariant under rotations, translations, and dilations. When we consider the finiteness problem, it is natural to restrict our counting to the classes of central configurations modulo these transformations.

Let $x=(x_{1},x_2,x_3,x_3,x_{5},x_{6})$ be a symmetric planar central configuration  of the Newtonian six-body problem in $\mathbb{R}^2$  for which there are  four bodies on a symmetry line $s$ of the set $X=\{x_1,x_2,x_3,x_3,x_5,x_6\}$. This configurations will be named \emph{cross central configurations} of the six-body problem. For short we use the notation CC6BP. We index the bodies so that $x_1$, $x_2$, $x_3$ and $x_4 \in s$.  In this way the bodies $x_5$ and $x_6$ belong to a line $l$ perpendicular to $s$.

Take a orthogonal system of coordinates for $\mathbb{R}^2$ such that the line $l$ is parallel to  the $x$-axis.  The coordinates of the bodies $x_i$ are given by  $(x_{i1},x_{i2})$. We can to apply suitable homothety and  rotation to the configuration $x$, and re-index the bodies, if it is necessary, so that the following conditions are satisfied by a CC6BP: 

\begin{enumerate}
\item[i.] $x_{12}>x_{22}>x_{32}>x_{42}$;
\item[ii.] $x_{52}>x_{32}$;
\item[iii.] $x_{52}-x_{42}=1$;
\item[iv.] $x_{51}<x_{11}$ and $x_{61}>x_{11}.$
\end{enumerate}

Let $\mathcal{X}$ be the set of the cross central configurations of the six-body problem with the center of mass fixed in the origin of the coordinate system and satisfying the four conditions above. Note that $\mathcal{X}$ contains one representative of  each class of  cross central configurations modulo translations, rotations and dilations. 

We can use the geometric constraints satisfied for a cross central configurations $x=(x_1,...,x_6)\in \mathbb{R}^{12}$ to obtain polynomial equations for the mutual distances $r_{ij}$. In fact, given a CC6BP $x \in \mathcal{X}$, take the point $Q=s\cap l$.
\begin{figure}[h!]
\includegraphics[scale=0.6]{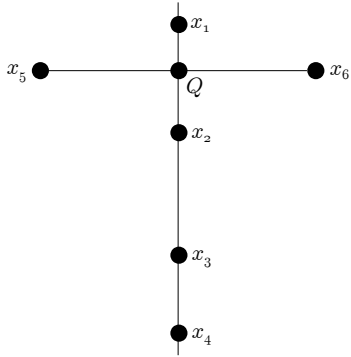}
\caption{An ilustration of CC6BP with its point $Q$.}\label{ccQ}
\end{figure}

From  the Pythagoras theorem for the triangles $x_1Qx_5$, $x_2Qx_5$, $x_3Qx_5$, $x_4Qx_5$, and the equations for the collinearity of the sets of points  $\{x_1,x_2,x_3\}$, $\{x_1,x_2,x_4\}$, $\{x_1,x_3,x_4\}$, $\{x_2,x_4,x_4\}$, it follows that the mutual distances between the bodies of a cross central configuration satisfies the following \emph{shape equations}:
\begin{align}\label{SH}
r_{12}+r_{23}-r_{13}&=0, \qquad & 4r_{15}^2-r_{56}^2-4(r_{14}-1)^2&=0           \\ \nonumber
r_{12}+r_{24}-r_{14}&=0, \qquad& 4r_{25}^2-r_{56}^2-4(1-r_{24})^2&=0,\\ \nonumber
r_{13}+r_{34}-r_{14}&=0, \qquad&4r_{35}^2-r_{56}^2-4(1-r_{34})^2&=0,\\ \nonumber
r_{23}+r_{34}-r_{24}&=0, \qquad& 4r_{45}^2-r_{56}^2-4&=0.\\ \nonumber
\end{align}

\section{Laura Andoyer Equations}
The Laura-Andoyer equations for planar and noncollinear central configurations with six bodies are the given by
\begin{equation}\label{LA}
L_{ij}=\sum_{k \neq i,j} m_{k}s_{ikj}(x_i-x_j)\wedge (x_i-x_k)=0,
\end{equation}
where $s_{ikj}=r_{ik}^{-3}-r_{jk}^{-3}$ and $1\leq i<l\leq 6$. It was proved in \cite[Ch.III]{HA} that the system \eqref{LA} is equivalent to the system of equations \eqref{eqcc} in the case of planar and noncollinear central configurations with center of mass at the origin of the coordinates system.
Note that 
\begin{align*}
(x_i-x_j)\wedge (x_i-x_k)&=\left|\begin{array}{cc}
x_{i1}-x_{j1}&x_{i1}-x_{k1}\\
x_{i2}-x_{j2}&x_{i2}-x_{k2}
\end{array}\right|e_1\wedge e_2\\
&=\left|\begin{array}{ccc}
1&1&1\\
x_{i1}&x_{j1}&x_{k1}\\
x_{i2}&x_{j2}&x_{k2}
\end{array}\right|e_1\wedge e_2.
\end{align*}
For $i,j,k$ different number from $1$ to $6$ define
\begin{equation}\label{Deltadef}
\Delta_{ijk}=\left|\begin{array}{ccc}
1&1&1\\
x_{i1}&x_{j1}&x_{k1}\\
x_{i2}&x_{j2}&x_{k2}
\end{array}\right|.
\end{equation}
If $\sigma$ is a permutation of $1,...,6$ we have
$$\Delta_{ijk}=(-1)^{\text{sgn}\sigma}\Delta_{\sigma_{i}\sigma_{j}\sigma_{k}}.$$ 
By the definition of the quantities $\Delta_{ijk}$ we can write the scalar Laura-Andoyer equations:
\begin{equation}\label{SLA}
L_{ij}=\sum_{k \neq i,j} m_{k}s_{ikj}\Delta_{ijk}=0,
\end{equation}
where $s_{ikj}=r_{ik}^{-3}-r_{jk}^{-3}$.
With these equations we can prove the following proposition:
\begin{prop}\label{massc} Given a CC6BP $x\in \mathcal{X}$  we have $m_5=m_6.$ 
\end{prop}
\begin{proof} From the collinearity of the bodies $x_1$, $x_2$, $x_3$ and $x_4$, we have   $\Delta_{123}=\Delta_{124}=\Delta_{134}=0$. The 
symmetry implies $r_{15}=r_{16}$, $\Delta_{125}=-\Delta_{126}$, and $\Delta_{135}=-\Delta_{136}$.  Replacing this relations in Laura-Andoyer equations $L_{12}$ and $L_{13}$ we obtain:
\begin{align*}
(m_{5}-m_6)s_{152}\Delta_{125}&=0,  \\ \nonumber
(m_{5}-m_6)s_{153}\Delta_{135}&=0.   \nonumber
 \end{align*}
By hypothesis $\Delta_{125} \neq 0$ and $\Delta_{135} \neq 0$, thus $m_5=m_6$ or $s_{152}=s_{153}=0.$ The last alternative implies $r_{15}=r_{25}=r_{35}$. This is impossible because the circumference of radius $r_{15}$ and center in the point $x_5$ intercepts the line $l$ defined by $x_{1}$ and $x_{2}$ in at least two distinct points. Hence $m_5=m_6.$ 
\end{proof}

 The symmetry conditions satisfied to a CC6BP $x\in \mathcal{X}$ and the proposition \ref{massc} implies that in this case there are only  four non-trivial Laura-Andoyer equations given by    
\begin{align}\label{4LA}
L_{15}=m_2s_{125}\Delta_{125}+m_3s_{135}\Delta_{135}+m_4s_{145}\Delta_{145}-m_5s_{165}\Delta_{156}=0;\\ \nonumber
L_{25}=m_1s_{215}\Delta_{125}-m_3s_{235}\Delta_{235}-m_4s_{245}\Delta_{245}+m_5s_{265}\Delta_{256}=0;\\ \nonumber
L_{35}=m_1s_{315}\Delta_{135}+m_2s_{325}\Delta_{235}-m_4s_{345}\Delta_{345}+m_5s_{365}\Delta_{356}=0;\\ \nonumber
L_{45}=m_1s_{415}\Delta_{145}+m_2s_{425}\Delta_{245}+m_3s_{435}\Delta_{345}+m_5s_{465}\Delta_{456}=0. \nonumber
\end{align}
If $x \in \mathcal{X}$, we have three cases for consider in terms of the relative position of $Q$ in relation to $x_1,x_2$, $x_3$ and $x_4$:
\begin{enumerate}
\item $x_{52}\geq x_{12}$;
\item $x_{12}>x_{52}\geq x_{22}$;
\item $x_{22}>x_{52} > x_{32}$.
\end{enumerate}
By the usual formula for the area of a triangle and the right-hand rule we  obtain, in all of these three cases above, that the quantities $\Delta_{ijk}$ and the mutual distances satisfies the following relations:
\begin{align}\label{rDelta}
\Delta_{125}&=-r_{12}\frac{r_{56}}{4}, &\qquad \Delta_{135}&=-r_{13}\frac{r_{56}}{4},\\ \nonumber
\Delta_{145}&=-r_{14}\frac{r_{56}}{4}, & \qquad \Delta_{156}&=(r_{14}-1)\frac{r_{56}}{2},\\ \nonumber
 \quad \Delta_{235}&=-r_{23}\frac{r_{56}}{4},&\qquad  \Delta_{245}&=-r_{24}\frac{r_{56}}{4},\\ \nonumber
\Delta_{256}&=-(1-r_{24})\frac{r_{56}}{2}, &\qquad \Delta_{345}&=-r_{34}\frac{r_{56}}{4},\\ \nonumber
\Delta_{356}&=-(1-r_{34})\frac{r_{56}}{2}, & \qquad\Delta_{456}&=-\frac{r_{56}}{2}.& & \nonumber
\end{align}
The relations \eqref{rDelta} allows us to eliminate the $\Delta_{ijk}$ variables of the Laura-Andoyer system. On dividing the resulting
system by $r_{56}\neq 0$ we get
\begin{align}\label{4LA}
L_{15}&=m_2s_{125}r_{12}+m_3s_{135}r_{13}+m_4s_{145}r_{14}+2m_5s_{165}(r_{14}-1)=0;\\ \nonumber
L_{25}&=-m_1s_{215}r_{12}+m_3s_{235}r_{23}+m_4s_{245}r_{24}-2m_5s_{265}(1-r_{24})=0;\\ \nonumber
L_{35}&=-m_1s_{315}r_{13}-m_2s_{325}r_{23}+m_4s_{345}r_{34}-2m_5s_{365}(1-r_{34})=0;\\ \nonumber
L_{45}&=-m_1s_{415}r_{14}-m_2s_{425}r_{24}-m_3s_{435}r_{34}-2m_5s_{465}=0.\\ \nonumber
\end{align}

\section{The Algebraic Variety of the CC6BP}\label{AVF}
In this section we  identify the CC6BP $x\in \mathcal{X}$ with points of a quasi-affine algebraic variety $V\subset \mathbb{C}^{32}$. We start introducing some terminology from basic algebraic geometry.

An \emph{affine algebraic variety} in the complex affine space $\mathbb{A}^n_{\mathbb{C}}$ is the common locus of a ideal $I\subset \mathbb{C}[x_1,...,x_n].$ We use the notation $V=Z(I)$. The topology on $\mathbb{A}^n_{\mathbb{C}}$  defined by the family of the complements of all algebraic varieties is called \emph{Zariski topology.} A quasi-affine algebraic variety $U$ is a relative open set of an algebraic variety. That is, $U=V\setminus W$, where $V$ and $W$ are algebraic varieties.

Consider the polynomial ring $A=\mathbb{C}[S_{ijk},R_{ij},M_i]$ in the following $32$ variables:
\begin{align*}
&S_{125},S_{135},S_{145},S_{165},S_{215},S_{235},S_{245},S_{265},\\
&S_{315},S_{325},S_{345},S_{365},S_{415},S_{425},S_{435},S_{465},\\
&R_{12},R_{13},R_{14},R_{15},R_{23},R_{24},R_{25},R_{34},R_{35},\\
&R_{45},R_{56},M_{1},M_{2},M_{3},M_{4},M_{5}
\end{align*}
 We will see $A$ as the ring of polynomial functions on $\mathbb{C}^{32}$. The points of $\mathbb{C}^{32}$ will be denoted by $P=(s_{ijk},r_{ij},m_i)$. The indexation of the entries of $P\in\mathbb{C}^{32}$ corresponds to the list of variables above. We use capital letters to refer  variables of a polynomial ring  and small letters to refer coordinates of a point to avoid ambiguities  in the notation. 
 
A point of $P=(s_{ijk},r_{ij},m_i)\in \mathbb{C}^{32}$ is associated to a  CC6BP $x\in \mathcal{X}$ with masses $m_1,m_2,m_3,m_4,m_5=m_6$ if the entries $r_{ij}$ consists of the mutual distances of $x$, $m_i$ are the masses of the bodies of $x$ and the quantities $s_{ikl}$ and $r_{ik}$ are related by the equations
\begin{equation}\label{sreq0}
s_{ijk}=r_{ij}^{-3}-r_{jk}^{-3}.
\end{equation}
In this case, we will say that \emph{P is a point of $\mathbb{C}^{32}$ associated to $x$}. Such points will be denoted by $P_{x}.$ Note that,   the mutual distances between the bodies of $x$ determine a unique point $P_x$ associated to $x$.

\begin{prop}\label{formulation} Every point $P_{x}\in\mathbb{C}^{32}$ associated to a CC6BP $x\in \mathcal{X}$ is in the quasi-affine algebraic set $\Omega= \widetilde{V}\setminus D$ where $\widetilde{V}=Z(I)$ is the zero locus of the ideal $I$  of $A$ generated by the following polynomials:
\begin{align*}
Z_{ij}&=R_{ij}^3R_{j5}^3S_{i5j}+R_{ij}^{3}-R_{j5}^{3},\quad i,j\in\{1,2,3,4\}, i\neq j,\\  \nonumber
W_{i}&=R_{i5}^3R_{56}^3S_{i65}+R_{i5}^{3}-R_{56}^{3},\quad i \in \{1,2,3,4\},\\  \nonumber
F_1&=R_{12}+R_{23}-R_{13},\\ \nonumber
F_2&=R_{12}+R_{24}-R_{14},\\ \nonumber
F_3&=R_{13}+R_{34}-R_{14},\\ \nonumber
F_4&=R_{23}+R_{34}-R_{24},\\ \nonumber
F_5&=4R_{15}^2-R_{56}^2-4(R_{14}-1)^2, \\ \nonumber
F_6&=4R_{25}^2-R_{56}^2-4(1-R_{24})^2,\\ \nonumber
F_7&=4R_{35}^2-R_{56}^2-4(1-R_{34})^2, \\ \nonumber
F_8&=4R_{45}^2-R_{56}^2-4,\\ \nonumber
L_{1}&=M_2S_{125}R_{12}+M_3S_{135}R_{13}+M_4S_{145}R_{14}+2M_5S_{165}(R_{14}-1),\\ \nonumber
L_{2}&=-M_1S_{215}R_{12}+M_3S_{235}R_{23}+M_4S_{245}R_{24}-2M_5S_{265}(1-R_{24}),\\ \nonumber
L_{3}&=-M_1S_{315}R_{13}-M_2S_{325}R_{23}+M_4S_{345}R_{34}-2M_5S_{365}(1-R_{34}),\\ \nonumber
L_{4}&=-M_1S_{415}R_{14}-M_2S_{425}R_{24}-M_3S_{435}R_{34}-2M_5S_{465},\\  \nonumber
\end{align*}
and  $D=\mathcal{Z}(R_{12},R_{13},R_{14},R_{15},R_{23},R_{24},R_{25},R_{34},R_{35},R_{45},R_{56}).$
\end{prop}
\begin{proof}
After clearing the denominators of the equations $\eqref{sreq0}$ we obtain that the equations
\begin{equation}\label{Sreq}
r_{ij}^3r_{jk}^3s_{ijk}+r_{ij}^{3}-r_{jk}^{3}=0
\end{equation}
are satisfied by the quantities $s_{ijk}$ and $r_{ij}$ associated to a CC6BP $x$. By the equations $\eqref{SH}$, $\eqref{4LA}$ and $\eqref{Sreq}$ we get that every point $P_x$ associated to a configuration $x\in \mathcal{X}$ belongs to $\widetilde{V}$. In the other hand, since the every mutual distance between bodies of $x$ is nonzero, $P_x \not \in D.$ The result is proved.
\end{proof}

\section{An example of CC6BP}\label{anexample}
In this section, we obtain a particular example of CC6BP that will be very important in order to prove our generic finiteness result.  The basic idea consists in finding the possible values for the mutual distance $r_{12}$. From this we find the other mutual distances correspondent to a cross central configuration. The main tool here is the extension theorem that we present in the following (See \cite[Ch.III]{CLO} for more details). 

Let  $I = \langle f_1,...,f_s \rangle \subset \mathbb{C}[X_1, . . . , X_n]$ be an ideal. The $l$-th elimination ideal $I_l$ is
the ideal of $\mathbb{C}[X_{l+1},...,X_n]$ defined by $I_{l} = I \cap \mathbb{C}[X_{l+1},...,X_n]$. A partial solution of the polynomial system $V(I)$ is a solution $(a_{l+1},...,a_{n})\in V(I_{l})$. If $(a_1,...,a_n)$ is a solution of $V(I)$ then obviously $(a_{l+1},...,a_{n})\in V(I_{l})$. The converse is not necessarily true. The next describes precisely when a partial solution can be extended to a complete solution.

\begin{theorem}[Extension Theorem]\label{ET}
Let $I =  \langle f_1,...,f_s \rangle \subset \mathbb{C}[X_1,...,X_n]$ be an ideal and
let $I_1$ be the first elimination ideal of $I$. For each $1 \leq i \leq s$, write $f_i$ in the form
$f_i = g_i(X_2,...,X_n)X_{1}^{N_i}$
 + terms in which $x_1$ has degree $< N_i$, $N_i \geq 0$ and $g_i \in \mathbb{C}[X_2,...,X_n]$ is nonzero. Suppose that we have a partial solution
$(a_2,...,a_n) \in V(I_1)$. If $(a_2,...,a_n)\not \in V(g_1,...,g_s)$, then there exists $a_1 \in \mathbb{C}$
such that $(a_1, a_2,...,a_n) \in V(I)$.
\end{theorem}

Now we pass to find our desired example.

\begin{prop}\label{importantexample}
 Assume $m_1=m_2=m_3=m_4=1$ and  $m_5=m_6$. Then,  there is a unique $x$ is a CC6BP satisfying the following conditions:
\begin{enumerate}
\item $x_1,x_2,x_3,x_4$ are collinear;
\item The polygon with vertices $x_2$, $x_3$, $x_5$ and $x_6$ is a square  centered in the origin of the coordinate system, such that $x_2,x_3$ are in the $y$-axis and $x_5,x_6$ are in the $x$-axis.
\item $r_{12}=r_{34}$ and $r_{14}=2$. 
\end{enumerate} 
  
\end{prop}
\begin{proof}  By hypotheses the values for the mutual distances of $x$ are:
\begin{align*}
r_{13}&=2-r_{12},&r_{14}&=2,& r_{23}&=2-2r_{12},\\ 
r_{24}&=2-r_{12}, &r_{34}&=r_{12}, &r_{35}&=r_{25},\\
r_{45}&=r_{15}, &r_{56}&=2-2r_{12}. 
\end{align*}
Substitute this relations and the values $m_2=m_3=m_5=m_6=1$ on the polynomial equations described on proposition $\ref{formulation}$. Next eliminate the variables $S_{ijk}$ from the resulting system by using the rational equations $S_{ijk}=R_{ij}^{-3}-R_{jk}^{-3}$. Following this two steps we have the following system:
\begin{align*}
R_{25}^2 - 2(1-R_{12})^2=0;\\
R_{15}^2 -1-(1-R_{12})^2=0;\\
\frac{1}{4} \, M_{5} {\left(\frac{1}{{\left(R_{12} - 1\right)}^{3}} + \frac{8}{R_{15}^{3}}\right)} + {\left(R_{12} - 2\right)} {\left(\frac{1}{{\left(R_{12} - 2\right)}^{3}} + \frac{1}{R_{25}^{3}}\right)} \\ \nonumber  
+ R_{12} {\left(\frac{1}{R_{12}^{3}} - \frac{1}{R_{25}^{3}}\right)} - \frac{2}{R_{15}^{3}} + \frac{1}{4}=0;\\ \nonumber
-\frac{1}{4} \, M_{5} {\left(R_{12} - 1\right)} {\left(\frac{1}{{\left(R_{12} - 1\right)}^{3}} + \frac{8}{R_{25}^{3}}\right)} + \frac{1}{4} \, {\left(R_{12} - 1\right)} {\left(\frac{1}{{\left(R_{12} - 1\right)}^{3}} + \frac{8}{R_{25}^{3}}\right)}\\ \nonumber
 + {\left(R_{12} - 2\right)} {\left(\frac{1}{{\left(R_{12} - 2\right)}^{3}} + \frac{1}{R_{15}^{3}}\right)} - R_{12} {\left(\frac{1}{R_{12}^{3}} - \frac{1}{R_{15}^{3}}\right)}=0. \nonumber
\end{align*}
After clearing the denominators of the last two polynomial equations in the system above we obtain the ideal $I\subset\mathbb{C}[M_5,R_{15},R_{25},R_{12}]$ generated by:
\begin{align*}\
&F_1=R_{25}^2 - 2(1-R_{12})^2;\\
&F_2=R_{15}^2 -1-(1-R_{12})^2;\\
&F_{3}=R_{12}^{7} R_{15}^{3} R_{25}^{3} - 7 \, R_{12}^{6} R_{15}^{3} R_{25}^{3} + 8 \, M_{5} R_{12}^{7} R_{25}^{3} + M_{5} R_{12}^{4} R_{15}^{3} R_{25}^{3}\\
& + 27 \, R_{12}^{5} R_{15}^{3} R_{25}^{3} - 8 \, R_{12}^{7} R_{15}^{3}
 - 56 \, M_{5} R_{12}^{6} R_{25}^{3} - 8 \, R_{12}^{7} R_{25}^{3}\\ 
 &- 4 \, M_{5} R_{12}^{3} R_{15}^{3} R_{25}^{3} - 65 \, R_{12}^{4} R_{15}^{3} R_{25}^{3} + 56 \, R_{12}^{6} R_{15}^{3} + 152 \, M_{5} R_{12}^{5} R_{25}^{3}\\
& + 56 \, R_{12}^{6} R_{25}^{3} + 4 \, M_{5} R_{12}^{2} R_{15}^{3} R_{25}^{3} + 104 \, R_{12}^{3} R_{15}^{3} R_{25}^{3} - 152 \, R_{12}^{5} R_{15}^{3}\\
& - 200 \, M_{5} R_{12}^{4} R_{25}^{3} - 152 \, R_{12}^{5} R_{25}^{3} - 108 \, R_{12}^{2} R_{15}^{3} R_{25}^{3} + 200 \, R_{12}^{4} R_{15}^{3}\\ 
&+ 128 \, M_{5} R_{12}^{3} R_{25}^{3} + 200 \, R_{12}^{4} R_{25}^{3} + 64 \, R_{12} R_{15}^{3} R_{25}^{3} - 128 \, R_{12}^{3} R_{15}^{3}\\
& - 32 \, M_{5} R_{12}^{2} R_{25}^{3} - 128 \, R_{12}^{3} R_{25}^{3} - 16 \, R_{15}^{3} R_{25}^{3} + 32 \, R_{12}^{2} R_{15}^{3} + 32 \, R_{12}^{2} R_{25}^{3};\\ 
&F_{4}= -8 \, M_{5} R_{12}^{7} R_{15}^{3} - M_{5} R_{12}^{4} R_{15}^{3} R_{25}^{3} + 56 \, M_{5} R_{12}^{6} R_{15}^{3} + 8 \, R_{12}^{7} R_{15}^{3}\\
& + 8 \, R_{12}^{7} R_{25}^{3} + 4 \, M_{5} R_{12}^{3} R_{15}^{3} R_{25}^{3} + R_{12}^{4} R_{15}^{3} R_{25}^{3} - 152 \, M_{5} R_{12}^{5} R_{15}^{3}+ R_{12}^{6} R_{15}^{3}\\
& - 56 \, - 56 \, R_{12}^{6} R_{25}^{3} - 4 \, M_{5} R_{12}^{2} R_{15}^{3} R_{25}^{3} + 12 \, R_{12}^{3} R_{15}^{3} R_{25}^{3} + 200 \, M_{5} R_{12}^{4} R_{15}^{3}\\
& + 152 \, R_{12}^{5} R_{15}^{3} + 152 \, R_{12}^{5} R_{25}^{3} - 44 \, R_{12}^{2} R_{15}^{3} R_{25}^{3} - 128 \, M_{5} R_{12}^{3} R_{15}^{3}\\
& - 200 \, R_{12}^{4} R_{15}^{3} - 200 \, R_{12}^{4} R_{25}^{3} + 48 \, R_{12} R_{15}^{3} R_{25}^{3} + 32 \, M_{5} R_{12}^{2} R_{15}^{3}\\
& + 128 \, R_{12}^{3} R_{15}^{3} + 128 \, R_{12}^{3} R_{25}^{3} - 16 \, R_{15}^{3} R_{25}^{3} - 32 \, R_{12}^{2} R_{15}^{3} - 32 \, R_{12}^{2} R_{25}^{3}.
\end{align*}
If $x$ is a CC6BP satisfying the conditions (1), (2) and (3) the correspondent point $(m_5,r_{12},r_{15},r_{25})$ are in the zero locus of $I$ in $\mathbb{C}^4.$ Note that in this case $0<r_{12}<1$. We will conclude the argument by using the extension theorem. We start obtaining the elimination ideal $I_{3}=I \cap \mathbb{C}[r_{12}].$ Making computations with the software Singular we obtain that $I_3$ is generated by an unique polynomial $g(R_{12})=(R_{12}-1)^4h(R_{12})$, for which $1$ is not root of 
\begin{align*}
&h(R_{12})=49R_{12}^{52}-2548R_{12}^{51}+66738R_{12}^{50} + \text{terms of lower degree}
\end{align*} 
Using the Sturm theorem it is possible to check that $h$ has only one root $\tilde{r}_{12}$ in the interval $(0,1]$. More precisely,
we can use the Sturm theorem again to ensure $\tilde{r}_{12}\in (0.4402418528, 0.4402418529]$.
Consider the elimination ideals $I_2=I \cap \mathbb{C}[R_{25},R_{12}]$ ans $I_{1}=\mathbb{C}[R_{15},R_{25},R_{12}]$.
Since  $R_{25}^2 - 2(1-R_{12})^2\in I_{2}$ and $R_{15}^2 -1-(1-R_{12})^2 \in I_1$, by the extension theorem 
$$P_{0}=\left(\sqrt{1+(1-\tilde{r}_{12})^2},\sqrt{2(1-\tilde{r}_{12})^2},\tilde{r}_{12}\right) \in Z(I_1)$$
is a partial solution. Finally, in order to extend a partial solution $P_0$ to a solution, we must to proof that there is $m_5\in \mathbb{R^{+}}$ such that 
$$\left(m_5,\sqrt{1+(1-\tilde{r}_{12})^2},\sqrt{2(1-\tilde{r}_{12})^2},\tilde{r}_{12}\right) \in Z(I).$$
By the extension theorem it is sufficient that the leading coefficients of $F_3$ and $F_4$ written as 
polynomials of $\mathbb{C}[R_{15},R_{25},R_{12}][M_{5}]$ does not vanish simultaneously at the partial solution $P_{0}$. Such leading terms are given by:
\begin{align*}
LC_1=&8R_{12}^7R_{25}^3-56R_{12}^6R_{25}^3+152R_{12}^5R_{25}^3+R_{12}^4R_{15}^3R_{25}^3\\
&-200R_{12}^4R_{25}^3-4R_{12}^3R_{15}^3R_{25}^3+128R_{12}^3R_{25}^3\\
&+4R_{12}^2R_{15}^3R_{25}^3-32R_{12}^2R_{25}^3;
\end{align*}
\begin{align*}
LC_2=&-8R_{12}^7R_{15}^3+56R_{12}^6R_{15}^3-152R_{12}^5R_{15}^3-R_{12}^4R_{15}^3R_{25}^3\\
&+200R_{12}^4R_{15}^3+4R_{12}^3R_{15}^3R_{25}^3-128R_{12}^3R_{15}^3\\
&-4R_{12}^2R_{15}^3R_{25}^3+32R_{12}^2R_{15}^3.
\end{align*}

Let $$P_t=\left(\sqrt{1+(1-t_{12})^2},\sqrt{2(1-t_{12})^2},t_{12}\right)$$
be the truncated partial solution. First we will evaluate the leading terms $LC_1$ and $LC_2$ at $P_t$.

 Using the function roots() of the software SageMath we obtain the truncated value $t_{12}=0.440241852870668$ for  the solution $\tilde{r}_{12}$ with precision of $10$ decimal cases.  A direct computation shows $LC_1(P_t)=0.0238525134676166$ and $LC_2(P_t)=0.643697010003912.$ 

Next we use the truncated values $LC_1(P_t)$ and $LC_2(P_t)$ for estimate  $LC_1(P_0)$ and $LC_2(P_0)$ using elementary calculus.
Consider the differentiable functions on $R_{12}$:
   
$$\alpha_i(R_{12})=LC_1\left(\sqrt{1+(1-R_{12})^2},\sqrt{2(1-R_{12})^2},R_{12}\right).$$

 The middle value theorem implies that there are constants $c_i$ in the interval defined by $t_{12}$ and $\tilde{r}_{12}$ such that
$$|\alpha_i(\tilde{r}_{12})-\alpha_i(t_{12})|=|\alpha^\prime_{i}(c_i)|\tilde{r}_{12}-t_{12}|.$$
 Sturm Theorem ensure that $|t_{12}-\overline{r}_{12}|<10^{-10}$. Combining this inequality, with the triangular rule and the estimate $c_i<1/2$, we can assert that
 $$|LC_i(P_0)-LC_i(P_t)|=|\alpha_i(\tilde{r}_{12})-\alpha_i(t_{12})|<10^{-5}$$ 
  for $i=1,2$. In particular,  $LC_{1}(P_0)\neq 0$ and $LC_2(P_0)\neq 0$. 
  Since the polynomials $F_3$ and $F_4$ are linear on $M_5$, there exists unique solution $$(m_5,r_{15},r_{25},r_{12})\in Z(I)$$ that provides a CC6BP.  The truncated value for $m_5$ is $4.76482836$. The result is proved.
\end{proof}

\begin{figure}[h!]
\includegraphics[scale=0.6]{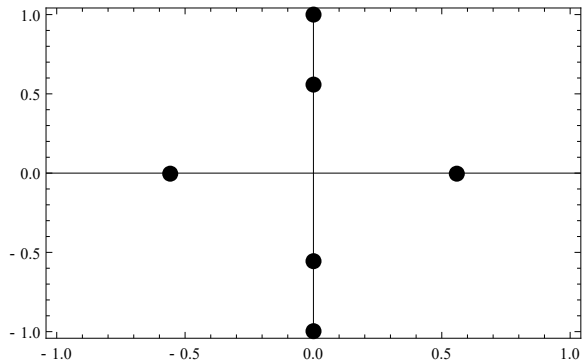}
\caption{The example of CC6BP founded in proposition \ref{importantexample}}
\end{figure}
\section{The Jacobian Criterion and Groebner basis theory} \label{ARG}

To obtain our finiteness result, we need to compute the dimension of the quasi-affine algebraic set $V$ defined in the proposition \ref{formulation}.  We will use the  Jacobian criterion and a computational procedure based in the Hilbert polynomial and Groebner basis computations. For completeness,  we introduce some important definitions used in our approach.

\subsection{Complex Algebraic Geometry}

In this part, we present basic definitions and results about algebraic varieties necessary to enunciate the Jacobian criterion and the dimension of fibers theorem. Our references here are \cite[Ch.I, Ch.II.1]{SH} and \cite[Ch.VI]{SKKT}.

Consider an affine algebraic variety $V$. A \emph{subvariety} of $V$ is a set of the form $W\cap V$ for which $W$ is an algebraic variety. If $V$ cannot be written as the nontrivial union of two subvarieties, it is  said to be \emph{irreducible.} We have that $V$ is irreducible if and only if the ideal of V, $\mathcal{I}(V)$ is primary.  It is true that every algebraic variety $V$ can be written as a finite union of irreducible subvarieties. In other words, there exist subvarieties $V_1,...,V_k$ of $V$ such that $V=V_1\cup...\cup V_k$. This fact will be important in our argument.  

 Given a subset $L\subset \mathbb{C}^n$ we define $\mathcal{I}(L)$ by the ideal of the polynomials vanishing on all elements of $L.$ It is easy to prove that  $Z(\mathcal{I}(L))=\overline{L}$ and $\mathcal{I}(Z(I))=\sqrt{I}$.  For every non-empty open subset $U$  of a irreducible variety $V$ we have that $U$ is dense in $V$ and $\mathcal{I}(U)=\mathcal{I}(V)$. We  say that a property is \emph{generically} on an irreducible algebraic  variety $V$, if it holds on a nonempty open subset.

The \emph{dimension} $\text{dim}(V)$ of a affine variety $V$ is defined to be the length $d$ of the longest
possible chain of proper irreducible subvarieties of $V$,
$$V=V_{d}\supsetneqq V_{d-1}\supsetneqq ... \supsetneqq V_1 \supsetneqq V_{0}.$$
It is easy to see that the dimension of $V$ is the maximum of the dimension of its irreducible components. For our purposes, the two most important facts about the dimension of an algebraic variety are:

\begin{itemize}
\item $\text{dim}(V)=0$ if, and only if, $V$ is finite.
\item If $V$ is irreducible and $U$ is non-empty subset of $V$ then $\text{dim}(V)=\text{dim}(U)$.
\end{itemize}


The \emph{tangent space} of $V=Z(F_1,...,F_{k})$ at $P$ is the linear variety 
$$\Theta_{P}V=Z(dF_{1}|_{P}(X-P),...,dF_{k}|_{P}(X-P))\subset \mathbb{A}^n_{\mathbb{C}}.$$

 The tangent space it is a local notion. More precisely, if $U$ is a non-empty open set of an affine algebraic variety $V$ and $P\in U$, then $\Theta_{P}V$ is isomorphic to $\Theta_{P}U$ as $\mathbb{C}$-vector spaces. Differently from differential manifolds, in an algebraic variety the dimension of the tangent space depends on the choice of a particular point $P \in V.$ For this reason, we need a punctual notion of dimension:

The dimension of $V$ at a point $P$ is given by 
$$\text{dim}_{P}(V)=\text{max}\{\text{dim}_{P}(V_{i}): V_{i} \text{ is a irreducible component of V}\}.$$ 

A point $P\in V$ is said to be \emph{nonsingular} if $\text{dim}_{P}(V)=\text{dim}(\Theta_{P}V).$

We need to relate the dimension of a variety $V$ with the dimension of the tangent space of $V$ at a point $P$. The next theorem will be useful for this task.

\begin{theorem}[Jacobian Criterion]\label{cj}
Let $V=Z(\langle f_{1},...,f_{m}\rangle)$ be a affine algebraic variety. A point $P\in V$ is a non-singular  point of $V$, if and only if, the rank of the Jacobian matrix $J(f_{1},...,f_{m})(P)$  at the point $P$ is given by $n-dim_{P}V$, for which,
$$J(f_{1},...,f_{m})(P)= \left(\begin{array}{ccc}
                    \frac{\partial{f_1}}{\partial x_1}(P) &\cdots& \frac{\partial{f_1}}{\partial x_n}(P)\\
                    \vdots                             & \vdots & \vdots\\
                     \frac{\partial{f_m}}{\partial x_n}(P)&  \cdots& \frac{\partial{f_m}}{\partial x_n}(P)\\
                    \end{array}\right).$$
Moreover, for every point $P \in V$ we have $\text{dim}_{P}(V)\leq \text{dim}(\Theta_{P}V).$  
\end{theorem}

Let $f: V\rightarrow W$ be a map between quasi-affine algebraic variety with $W\subset \mathbb{C}[X_1,...,X_n]$. $f$ is said regular in $P\in V$ if there is open subset $U \subset V$ containing $P$ and polynomials $g_1,...,g_n,h_1,...,h_n$ with $h_1(P),...,h_n(P)\neq 0$ and for all $Q\in U$ we have
$$f(Q)=\left(\frac{g_1}{h_1}(Q),...,\frac{g_n}{h_n}(Q)\right).$$
A regular map is said to be \emph{dominant} if $f(V)$ is dense in W. In this case $f(V)$ contains a non-empty open subset of $W.$

The canonical projections are simple examples of regular maps that will be used consistently in the  next section.  To finally this part, we enunciate another tool essential for our proof.

\begin{theorem}[Dimension of fibers Theorem]
Let $V$ and $W$ be irreducible quasi-affine algebraic varieties. Set $f: V \rightarrow W$ a surjective
regular morphism. Then, there exists a Zariski-open subset $U \subset W$ such that for each $P \in U$
we have that $\text{dim}(f^{-1}(P))=\text{dim}(V)-\text{dim}(W).$
\end{theorem}


\subsection{Groebner Basis}

In the proof of proposition $\ref{lemmaD3}$, we relax the concept of Groebner basis to compute a dimension of some irreducible components of the quasi-affine variety $V$ defined in \ref{formulation}. To explain our idea, we will introduce some concepts of Groebner basis theory. The basic reference for this part is \cite[Ch.IX.3]{CLO}.

A monomial ordering on $R=\mathbb{C}[X_1,...,X_n]$ is a total order on the set of the monomials of $R$, $\text{Mon}(R)=\{X^{a}=X_{1}^{a_{1}}...X_{n}^{a_{n}}:a=(a_1,...,a_n)\in \mathbb{N}^{n}\}$ that satisfies the following conditions:
\begin{enumerate}
\item $x^{a}\prec x^{b} \Rightarrow x^{a+c}\prec x^{b+c} $, for all $a,b,c\in\mathbb{N}\setminus \{0\};$
\item   $1\prec x^{a},$ for all $a \in \mathbb{N}\setminus \{0\}.$
\end{enumerate}

Fixed a monomial ordering $\prec$ on $R$ and chosen a polynomial $f\in R$ it is possible to write $f$ in the form
$$f=c_{a}X^{a}+\displaystyle \sum_{X^{b}\prec X^{a}}c_{b}X^{b},$$
for which $c_{b}\in \mathbb{C}$, $c_{a}\neq 0$. The term $\text{LT}(f)=c_{a}X^{a}$ is called \emph{leading term} of $f$. 

Given a ideal $I \subset R$, we define \emph{the ideal of leading terms of} $I$ by $\text{LT}(I)=\{\emph{LT}(f): f\in I\}$. A finite subset $G=\{g_1,...,g_k\}\subset I$ is named \emph{Groebner basis} if $\text{LT}(I)=\langle\text{LT}(g_1),...,\text{LT}(g_k)\rangle.$

Our next goal is to define the Hilbert polynomial of an ideal.  This concept is very useful  to determine the dimension of an algebraic variety. The basic idea is that the Hilbert polynomial of $I$ contains information about the ``size'' $I$, and consequently, it will provide the dimension of the variety $V=Z(I)$. The most natural idea for computing the ``size'' of $I$ is to compute the codimension of $I$ as a vector space of $\mathbb{C}[X_1,...,X_n]$. Since this vector spaces have infinite dimension, we need to work on the finite dimensional vector spaces that we will define in the following. 

Let $\mathbb{C}[X_1,...,X_n]_{\leq s}$ be the vector space of the polynomials with total degree less than or equal to $s$ and let $I_{\leq s}=I\cap\mathbb{C}[X_1,...,X_n]_{\leq s}$ be the vector space of the polynomials in $I$ with degree less than or equal to $s$. The  \emph{affine Hilbert function} of the ideal $I\subset \mathbb{C}[X_1,...,X_n]$ is the function on the nonnegative integers $s$ defined by
$$HF(I)(s)=\text{dim}\left( \mathbb{C}[X_1,...,X_n]_{\leq s}\right)-\text{dim}(I_{\leq s}).$$

The most basic property of Hilbert function is that for $s$ sufficiency large the Hilbert function $\text{HF}_{I}(s)$ can be written as a polynomial $HP_{I}(s)=\sum_{i=0}^{d}b_{i}\binom{s}{d-i}$ for which $b_i$ are integers and $b_{0}$ is positive. This polynomial is called the \emph{the Hilbert polynomial} of $I\subset \mathbb{C}[X_1,...,X_{n}]$. 

A monomial ordering $\prec$ on $\mathbb{C}[X_1,...,X_n]$ is a graded order if $X^b \prec X^a$ whenever
$|b_1+...+b_n| < |a_1+...+a_n|$. Consider an ideal $I\subset\mathbb{C}[X_1,...,X_n]$ and take $\text{LT}(I)$ the ideal of leading terms of $I$ with respect some graded monomial ordering. It is very important to note that in this hypothesis the  Hilbert polynomials of $I$ and $\text{LT}(I)$ are equal. The next theorem explicit the relation between the Hilbert polynomial and dimension.

\begin{theorem}[The Dimension Theorem] Suppose that $\prec$ is a graded monomial ordering on $\mathbb{C}[X_1,...,X_n].$
Let $I$ and ideal of $\mathbb{C}[X_1,...,X_n]$ and $\text{LT}(I)$ the ideal of leading terms of $I$ with respect to $\prec$.
Then,
$$\text{dim}(Z(I))=\text{deg}(\text{HP}_{I})=\text{deg}(\text{HP}_{\text{LT}(I)}).$$
\end{theorem}

Now we will enunciate a simple observation that will be the basis for the procedure described in proposition $\ref{lemmaD3}$.

\begin{lemma}\label{PLT}
Let $I$  be a  ideal of $\mathbb{C}[X_1,...,X_n]$ and $F=\{f_1,...,f_k\}$ a set of monomials such that $F\subset \text{LT}(I)$, for which $\text{LT}(I)$ is the ideal of leading terms of $I$ with respect to a graded monomial order $\prec$ on $\mathbb{C}[X_1,...,X_n]$. If $\text{dim}(Z(F))=k$ then $\text{dim}(Z(I)))\leq k.$ 
\end{lemma}
\begin{proof}
Since $F\subset LT(I)$ then $Z(\text{LT}(I))\subset Z(F)$. By the Dimension Theorem, 
$$\text{dim}(Z(I)))=\text{dim}(Z(\text{LT}(I))))\leq\text{dim}(Z(F))=k.$$
\end{proof}

\section{The Finiteness Result}\label{FN}

We call $\mathbb{C}^{11}$ the \emph{mutual distances space} and $\mathbb{C}^5$ the \emph{mass space}.
Consider the projections $\pi_1:\Omega \rightarrow \mathbb{C}^{11}$   and $\pi_2:\Omega \rightarrow \mathbb{C}^5$ defined by 
\begin{equation*}
\pi_1((r_{ij},s_{ijk},m_i))=(r_{ij}) \text{ and }  \pi_2((r_{ij},s_{ijk},m_i))=(m_i). 
\end{equation*}
We will study this canonical projections using the tools developed in the last section.  Our approach let us to avoid the Groebner basis computations involving directly the Laura-Andoyer equations.

\begin{prop}\label{P1} Consider the affine varieties $H=Z(F_1,...,F_8)$ and $D$ of $\mathbb{C}^{11}$, for which $F_1$,...,$F_8$ and $D$ are as in the proposition  \ref{formulation}. Define the quasi-affine variety $E=H \setminus D$. Then  $\overline{\pi_1(\Omega)} \subset E$.
\end{prop}
\begin{proof}
The polynomials $F_1,...,F_8$ are written only in terms of the variables $R_{ij}$ and belong to the ideal of $\Omega$.
Then, $F_1,...,F_8 \in \mathcal{I}(\pi_1(\Omega))$ and consequently $\overline{\pi_1(\Omega)}\subset H$. By proposition \ref{formulation}, $\Omega=V\setminus D$, which gives $\overline{\pi_1(\Omega)}\cap D = \emptyset$. Therefore, $\overline{\pi_1(\Omega)} \subset E=H\setminus D$.
\end{proof}

By proposition $\ref{P1}$ we can restrict our study to the projection on the mutual distances space to $\pi_1: \Omega \rightarrow E$. Making computations in the software Singular we obtain the following:
\begin{enumerate}
\item $\text{dim}(H)=4$;
\item $\mathcal{I}(H)$ is primary ideal.
\end{enumerate}
Consequently,  $E=H\setminus D$ is a irreducible quasi-affine algebraic set and $\text{dim}(E)=4$. 
Let $\bar{r}=(\bar{r}_{ij})$ be an arbitrary point of $\pi_1(\Omega)$. The fiber $\pi_1^{-1}(\bar{r}_{ij})$ is a closed subset of $\Omega$ defined
by the polynomial equations $A_{ij}=R_{ij}-\bar{r}_{ij}=0.$ Note that the polynomials $B_{ijk}=\bar{r}_{ij}^3\bar{r}_{jk}^3S_{ijk}+\bar{r}_{ij}^3-\bar{r}_{jk}^3\in \mathcal{I}(\pi_1^{-1}(\bar{r}_{ij}))$. Hence,
$$\pi_1^{-1}(\bar{r}_{ij})= Z(A_{ij},B_{ijk},L_i)\setminus D,$$
for which the polynomials $L_i$ and the algebraic variety $D$ are as in  proposition \ref{formulation}.

We will provide an estimate for the dimension of the fibers using the Jacobian criterion. Let $P=(s_{ijk},\bar{r}_{ij},m_i)\in \pi_1^{-1}(\bar{r}_{ij})$. The Jacobian matrix $J(P)=J(B_{ijk},A_{ij},L_i)$ is given by
\renewcommand{\arraystretch}{1.5}
 $$J(P)=\left(\begin{array}{c|c|c}
 \frac{\partial B_{ijk}}{\partial S_{ijk}}_{16 \times 16}& \frac{\partial B_{ijk}}{\partial R_{ij} }_{16 \times 11}&\frac{\partial B_{ijk}}{\partial M_i}_{16 \times 5}\\ \hhline{-|-|-}
 \frac{\partial A_{ij}}{\partial S_{ijk}}_{11 \times 16} &\frac{\partial A_{ij}}{\partial R_{ij} }_{11 \times 11}& \frac{\partial A_{ij}}{\partial M_i}_{11 \times 5}\\   \hhline{-|-|-}
\frac{\partial L_i}{\partial S_{ijk}}_{4\times 16}& \frac{\partial L_i}{\partial R_{ij} }_{4\times 11} &\frac{\partial L_i}{\partial M_i}_{4 \times 5}
 \end{array}\right)(P).$$
It is easy to see that, for all $P\in \pi_{1}^{-1}(\bar{r}_{ij})$, we have: 
\begin{align*}
\left[\frac{\partial B_{ijk}}{\partial S_{ijk}}\right](P)=[K]_{16 \times 16}& \quad  \left[\frac{\partial B_{ijk}}{\partial R_{ij}}\right](P)=[0]_{16 \times 11}& \left[\frac{\partial B_{ijk}}{\partial M_i}\right](P)=[0]_{16 \times 5}\\
 \left[\frac{\partial A_{ij}}{\partial S_{ijk}}\right](P)=[0]_{11 \times 16} & \quad \left[\frac{\partial A_{ij}}{\partial R_{ij} }\right](P)=[I]_{11 \times 11}& \quad \left[\frac{\partial A_{ij}}{\partial M_i}\right](P)=[0]_{11 \times 5}
\end{align*}
%


where $[I]_{11\times 11}$ denotes the identity matrix $11\times 11$, $[0]_{m\times n}$ denotes the null matrix $m\times n$ and 
$[K]_{16\times 16}$ is a diagonal non-singular matrix.

 Hence, the Jacobian matrix is block-triangular and the by Jacobian criterion we get 
\begin{equation}\label{dimfiber}
\text{dim}_{P}(\pi_{1}^{-1}(\bar{r}_{ij}))\leq 5- \text{rank}\left(\left[\frac{\partial L_i}{\partial m_i}\right](P)\right),
\end{equation}
for all $P\in \pi_{1}^{-1}(\bar{r}_{ij})$.\\

The next lemma provides an upper bound for the fibers of $\pi_1.$
\begin{prop}\label{dimf3}Consider the projection $\pi_1:\Omega\rightarrow E.$ If $ (\bar{r}_{ij})\in E$ is a vector of mutual distances provided by a CC6BP then $\text{dim}(\pi_{1}^{-1}(\bar{r}_{ij}))\leq 3$. 
\end{prop}
\begin{proof}
By inequality $\eqref{dimfiber}$, it is sufficient to proof that 
$$\text{rank}\left(\left[\frac{\partial L_i}{\partial m_i}\right](P)\right)\geq 2$$
for all $P=(s_{ijk},\bar{r}_{ij},m_i) \in \pi_{1}^{-1}(\bar{r}_{ij})$. $\left[\frac{\partial L_i}{\partial M_i}\right](P)$ is given by
$$\left(\begin{array}{cccc}
0 & - s_{215} \bar{r}_{12} & - s_{315} \bar{r}_{13} & s_{415} \bar{r}_{14} \\
s_{125} \bar{r}_{12} & 0 & - s_{325} \bar{r}_{23} & s_{425} \bar{r}_{24} \\
s_{135} \bar{r}_{13} & s_{235} \bar{r}_{23} & 0 & s_{435} \bar{r}_{34} \\
s_{145} \bar{r}_{14} & s_{245} \bar{r}_{24} & s_{345} \bar{r}_{34} & 0 \\
2s_{165}(\bar{r}_{14} - 1) & 2 s_{265}(\bar{r}_{24} - 1) & 2 s_{365}(\bar{r}_{34} - 1)& -2s_{465}(\bar{r}_{34} -1)
\end{array}\right).$$
We will divide the analysis in two cases. The first one is $s_{125}\neq 0$ or $s_{215}\neq 0.$ The shape of a $CC6BP$ implies  $\bar{r}_{14}>\bar{r}_{13}$ and $\bar{r}_{24}>\bar{r}_{23}$. Hence $s_{315}\neq 0$ or $s_{415}\neq0$ and
$s_{325}\neq 0$ or $s_{425}\neq0$. This considerations implies that the rank of 
$$\left(\begin{array}{cccc}
0 & - s_{215} \bar{r}_{12} & - s_{315} \bar{r}_{13} & s_{415} \bar{r}_{14} \\
s_{125} \bar{r}_{12} & 0 & - s_{325} \bar{r}_{23} & s_{425} \bar{r}_{24} \end{array}\right)$$
is greater than or equal to 2. \\

For conclude we  suppose now $s_{125}=s_{215}= 0$. Since $\bar{r}_{14}>\bar{r}_{24}$ and $\bar{r}_{14}>\bar{r}_{13}$, then $s_{315}\neq 0$ or $s_{415}\neq 0$  and  $s_{145}\neq 0$ or $s_{245}\neq 0.$ Hence the rank of
$$\left(\begin{array}{cccc}
0 & 0& s_{315} \bar{r}_{13}           & s_{415} \bar{r}_{14}     \\
s_{145} \bar{r}_{14} & s_{245} \bar{r}_{24} & 0&0 \end{array}\right)$$
is greater than $2$.
\end{proof}

In order to use the dimension of fibers theorem we will consider the projections $\pi_1:\Omega_i \rightarrow E$, $i=1,...,l$ defined on the irreducible components of $\Omega$.  Since $\text{dim}(E)=4$, we have $5$ cases to consider: $\text{dim}(\overline{\pi_1(\Omega_i)})=0,1,2,3$ or $4$.


%
%

Let  $\Delta_j$ be the determinantal variety given by the zero locus of the minor $j\times j$ of $\left[\frac{\partial L_i}{\partial M_i}\right](P)$. 
 For compute the dimension of $\Omega,$ we will examine all possibilities of intersections between the $\Delta_{j}$'s and an irreducible component $\Omega_i$. By proposition \ref{dimf3}  the components such that $\Omega_i \subset \Delta_2$ does not contains points $P_x$ provided by a CC6BP. In this way, we will exclude such components of our analysis.  

%
%

\begin{lemma} \label{lemmaD2}
Let $1\leq k \leq 4$. If $\text{dim}(\overline{\pi_1(\Omega_i)})\leq k$ and $\Omega_i \not \subset \Delta_k$ then $\text{dim}(\Omega_i)\leq 5.$
\end{lemma}
\begin{proof}Since $\Omega_i \not \subset \Delta_k$, there is non-empty subset $\Omega_{i_{0}}=\Omega_i\setminus \Delta_{k}\subset \Omega_i$.  So, we  restrict our attention to the projection  $\pi_1: \Omega_{i_{0}} \rightarrow \overline{\pi_1(\Omega_{i_{0}})}.$

There is irreducible component $W_j$ of $\overline{\pi_1(\Omega_{i_{0}})}$ and open subset $\Omega^{'}_{i_{0}}$ of $\Omega_{i_{0}}$ such that 
$\pi_1: \Omega^{'}_{i_{0}} \rightarrow W_{j}$ is dominant.  It follows that there are open subsets 
$\Omega^{''}_{i_{0}}\subset \Omega^{'}_{i_{0}}$ and $W^{'}_{j}\subset  W_{j}$ such that $\pi_1: \Omega^{''}_{i_{0}} \rightarrow W^{'}_{j}$ is surjective. By the dimension of fibers theorem there exists open subset $U\subset W^{'}_{j}$ such that 
\begin{equation*}
\dim(\Omega^{''}_{i_{0}})-\text{dim}(W^{'}_{j})=\text{dim}(\pi_1^{-1}(\bar{r}_{ij}))
\end{equation*}
for all  $(\bar{r}_{ij}) \in U$. $\Omega^{''}_{i_{0}}\cap \Delta_k=\emptyset$ implies
$$\text{rank}\left(\left[\frac{\partial L_i}{\partial m_i}\right](P)\right)\geq k,$$
for all $P \in \Omega^{''}_{i_{0}}$. Hence inequality $\eqref{dimfiber}$ yields $\text{dim}_{P}(\pi_1^{-1}(\bar{r}_{ij}))\leq 5- k$ for all $(\bar{r}_{ij})\in U$ and $P \in \Omega^{''}_{i_{0}}$. Thus  $\text{dim}(\pi_1^{-1}(\bar{r}_{ij}))\leq 5-k$ for every  $(\bar{r}_{ij})\in U$. We have $\text{dim}(W^{'}_{j})=\text{dim}(W_{j})\leq \text{dim}(\overline{\pi_1(\Omega_i)})\leq k.$ Therefore,
$$\dim(\Omega_{i})=\dim(\Omega^{''}_{i_{0}})=\text{dim}(W^{'}_{j})+\text{dim}(\pi_1^{-1}(\bar{r}_{ij}))\leq k+(5-k)=5.$$\end{proof}

\begin{prop} \label{propD2}
If $\text{dim}(\overline{\pi_1(\Omega_i)})\leq 2$ then $\text{dim}(\Omega_i)\leq 5.$
\end{prop}
\begin{proof}
By proposition \ref{dimf3} we can suppose without generality that $\Omega_i \not \subset \Delta_{k}$, $k=1 \text{~or~} 2$. The result follows immediately of lemma \ref{lemmaD2}.
\end{proof}

\begin{lemma}\label{lemmaD3}
If $\Omega_i \subset \Delta_3$ then $\text{dim}(\overline{\pi_1(\Omega_i)})\leq 2$. In particular, $\text{dim}(\Omega_i)\leq 5.$ 
\end{lemma}
\begin{proof}
In this case $\mathcal{I}(\Delta_3)\subset \mathcal{I}(\Omega_i).$  The ideal $\mathcal{I}(\Delta_3)$
is generated by the $40$ minors of order $3$ of the Jacobian matrix $\left[\frac{\partial L_i}{\partial m_i}\right].$
For fixing ideas, we consider the minor $D^{123}_{123}$ correspondent to the rows $1,2,3$ and columns $1,2,3$:
$$D^{123}_{123}=\left|\begin{array}{rrr}
0 & - S_{215} R_{12} & - S_{315} R_{13} \\
S_{125} R_{12} & 0 & - S_{325} R_{23} \\
S_{135} R_{13} & R_{235} R_{23} & 0
\end{array}\right|.$$
The  computation of the determinant yields the formula 
\begin{align}\label{D123}
D_{123}^{123}&=- S_{125} S_{235} S_{315} R_{12} R_{13} R_{23} + S_{135} S_{215} S_{325} R_{12} R_{13} R_{23}.
\end{align}
Substituting the equations $S_{ijk}=R_{ij}^{-3}-R_{jk}^{-3}$ on $\eqref{D123}$ and after clearing the denominators of the resulting expression
we obtain the polynomial
$$D_1=(R_{25}^3-R_{12}^3)(R_{35}^3-R_{23}^3)(R_{15}^3-R_{13}^3)-(R_{35}^3-R_{13}^3)(R_{15}^3-R_{12}^3)(R_{25}^3-R_{23}^3).$$ 
Note that $D_1\in \mathcal{I}(\overline{\pi_i(\Omega_i)})$ since it depends only of the mutual distances and  $D_1\in \mathcal{I}(\Omega_i)$. Repeating this process for all generators of $\mathcal{I}(\Delta_3)$ we get the polynomials denoted by $D_1,...,D_{40}\in \mathcal{I}(\overline{\pi_i(\Omega_i)}).$ 
Consider the ideal $J$  generated by  $D_1,...D_{40}$ and the polynomials $F_1,...,F_8$ defined on proposition \ref{formulation}. Note that $J$  is contained $\mathcal{I}(\pi_i(\Omega_i))$. Hence $Z(\mathcal{I}(\overline{\pi_i(\Omega_i))})=\overline{\pi_i(\Omega_i)}\subset Z(J)$ implies
$$\text{dim}(\overline{\pi_i(\Omega_i)})\leq \text{dim}(Z(J)).$$

To this end it is sufficient to proof  $\text{dim}(Z(J))\leq 2$. The computation of a Groebner basis for the ideal $J$ is an arduous  task because it means to obtain a complete list of generators for the ideal $LT(J)$.  The main observation is that we do not need to compute a whole Groebner basis for the ideal $J$. We need to compute only a sufficient number of leading terms to obtain the desired upper bound for the dimension of $J$. In the following, we describe the simple procedure based on lemma \ref{PLT}   used to obtain a ``partial Groebner basis'' for de ideal $J$. 

Define the  ideals $J_i=\langle S_1,...,S_8,D_i \rangle$ for which $i=1,...,40.$ Computing the Groebner basis of $J_i$,  with respect to the degree reverse  lexicographic ordering, and collecting its leading terms  we obtain the ideal $K$ the following monomials   
\begin{align*}
&R_{23}, R_{13}, R_{12}, R_{45}^2, R_{34}^2, R_{24}^2, R_{14}^2, R_{24}R_{25}^5R_{35}^3,\\
& R_{14}R_{15}^5R_{35}^3, R_{14}R_{15}^5R_{25}^3, R_{25}^6R_{34}R_{35}^3, R_{15}^6R_{34}R_{35}^3,\\
& R_{15}^6R_{24}R_{25}^3, R_{24}R_{25}^5R_{34}R_{35}^2R_{56}^2, R_{14}R_{15}^5R_{34}R_{35}^2R_{56}^2,\\
& R_{14}R_{15}^5R_{24}R_{25}^2R_{56}^2, R_{15}^6R_{25}^4R_{34},R_{25}^6R_{35}^4R_{56}^2, R_{15}^6R_{35}^4R_{56}^2,\\
& R_{15}^6R_{25}^4R_{56}^2, R_{15}^6R_{25}^4R_{35}^2.
\end{align*} 

 We notice that it is possible to obtain this list of leading terms on a notebook with 16GB of memory in a few minutes.
It is easy to check with the software Singular that $\text{dim}(Z(K))=2$.  Since $K \subset LT(J)$, by the lemma \ref{PLT}
$$\text{dim}(\overline{\pi_1(\Omega_i)})\leq \text{dim}(Z(J))=\text{dim}(Z(LT(J)))\leq \text{dim}(Z(K))=2.$$

Observe that $\text{dim}(\overline{\pi_1(\Omega_i)})\leq 2$ and all components of $\Omega$ satisfies $\Omega_i \not \subset \Delta_{2}$. For finish this proof we apply  lemma \ref{lemmaD2} to conclude that $\text{dim}(\Omega_i)\leq 5.$  
\end{proof}


\begin{prop} \label{propD3}
If $\text{dim}(\overline{\pi_1(\Omega_i)})=3$ then $\text{dim}(\Omega_i)\leq 5.$
\end{prop}
\begin{proof}
By lemma \ref{lemmaD3}  $\Omega_i \not \subset \Delta_3$. The result follows of lemma \ref{lemmaD2}
\end{proof}

\begin{lemma}\label{crucial} There is $y$ a CC6BP such that 
$\text{rank}\left(\left[\frac{\partial L_i}{\partial m_i}\right]\right)(P_{y})=4$.
\end{lemma}
\begin{proof}
Let $y$ be the CC6BP that we find in the proposition \ref{importantexample} and $P_y$ the point of $\mathbb{C}^{32}$ associated to $y$. We will denote the mutual distances of the configuration $y$ by $\tilde{r}_{ij}$.

Consider the following submatrix of $\left[\frac{\partial L_i}{\partial M_i}\right]$:
$$A=\left(\begin{array}{rrrr}
0 & -S_{215} R_{12} & -S_{315} R_{13} & S_{415} R_{14} \\
S_{125} R_{12} & 0 & -S_{325} R_{23} & S_{425} R_{24} \\
S_{135} R_{13} & S_{235} R_{23} & 0 & S_{435} R_{34} \\
S_{145} R_{14} & S_{245} R_{24} & S_{345} R_{34} & 0
\end{array}\right).$$
Using the equations $S_{ijk}=R_{ij}^{-3}-R_{jk}^{-3}$ we can express the determinant  $|A|$ as a rational function in the mutual distances. The denominator of $|A|$ is given by $R_{12}^4R_{13}^4R_{14}^4R_{15}^3R_{23}^4R_{24}^4R_{25}^3R_{34}^4R_{35}^3R_{45}^3$, hence is non-zero when it is evaluated on the mutual distances associated to $y$. Let $\beta(R_{ij})$ the numerator of $|A|$ written as a real function of the $R_{ij}$ variables . In order to obtain the result it is sufficient to proof that $\beta(\tilde{r}_{ij})\neq 0.$

All the mutual distances associated to $y$ can be written only in terms of $\tilde{r}_{12}:$
\begin{align*}
&\tilde{r}_{13}=2-\tilde{r}_{12},\tilde{r}_{14}=2,\tilde{r}_{15}=\sqrt{1+(1-\tilde{r}_{12})^2}, \tilde{r}_{23}=2-2\tilde{r}_{12},\tilde{r}_{24}=2-\tilde{r}_{12},\\
&\tilde{r}_{25}=\sqrt{2(1-\tilde{r}_{12})^2}, \tilde{r}_{34}=\tilde{r}_{12},\tilde{r}_{35}=\tilde{r}_{25},\tilde{r}_{45}=\tilde{r}_{15},\tilde{r}_{56}=2-2\tilde{r}_{12}^2.
\end{align*}
Therefore, for our purposes, we can take $\beta$ as a function of $R_{12}$ defined in a small neighborhood of $\tilde{r}_{12}$. Let $t_{12}=0.440241852870668$ be the truncated value of $\tilde{r}_{12}$ with accuracy of $10$ decimal cases obtained in proposition \ref{importantexample}. The value of $\beta(t_{12})$ is exactly
$$\beta(t_{12})=11.2514100393576\sqrt{2} - 233.179777444682.$$
Using an argument based on the middle value theorem similar to the one given in the proposition \ref{importantexample} we can to prove that the error committed in the computation of $\beta(\tilde{r}_{ij})$ above is less than or equal to $10.$ Hence, $|A|\neq 0$ and   
$\text{rank}\left(\left[\frac{\partial L_i}{\partial m_i}\right]\right)(P_{y})=4$.
\end{proof}

\begin{prop}\label{propD4}
If $\text{dim}(\overline{\pi_1(\Omega_i)})=4$ then $\text{dim}(\Omega_i)\leq 5.$
\end{prop}
\begin{proof}
Consider the projection $\pi_1: \Omega_i \rightarrow E$. Since $\text{dim}(\overline{\pi_1(\Omega_i)})=4$, $\pi_1(\Omega_i)\subset E$,  and $\text{dim}(E)=4$ we get $\overline{\pi_1(\Omega_i)}=E$. In particular $\pi_1(\Omega_i)$ is dense on $E$ and it contains a non-empty open subset $\widetilde{U}$ of $E$.   Let $y$ be the CC6BP obtained in the proposition \ref{importantexample} and $\tilde{r}_{ij}$ its associated mutual distances and $P_{y}=\pi_{1}^{-1}(\tilde{r}_{ij}).$ Proposition \ref{crucial} implies  $\text{rank}\left(\left[\frac{\partial L_i}{\partial m_i}\right]\right)(P_{y})=4$, and consequently  $P_y \not \in \Delta_4$. Define the projection $\tilde{\pi}: \Delta_4\rightarrow \mathbb{C}^{11}$ and take the open set $U= \mathbb{C}^{11} \setminus\tilde{\pi}_{1}(\Delta_4)$. Note that $P_y \in E \cap U$, therefore the relative open set of $E$ given by $U^{'}=E \cap U$ is non-empty.  Since the intersection between two non-empty open subsets of a irreducible quasi-affine algebraic variety is ever non-empty, we obtain a point $(s_{ij})\in U^{'} \cap \widetilde{U}$. Observe that $\pi_1^{-1}(s_{ij})\in \Omega_i \setminus \Delta_4$. Hence $\Omega_i \not \subset \Delta_4$ and the result follows of lemma $\ref{lemmaD2}$.

\end{proof}

\begin{theorem}
$\text{dim}(\Omega) \leq 5.$
\end{theorem}
\begin{proof}
We classify the components $\Omega_i$ of $\Omega$ in terms of $\text{dim}(\overline{\pi(\Omega_i)}).$ For all cases 
propositions \ref{propD2}, \ref{propD3} and \ref{propD4} imply $\text{dim}(\Omega_i)\leq 5$. 
\end{proof}

\begin{theorem}
Consider the projection $\pi_2:\Omega \rightarrow \mathbb{C}^5$. There is a proper closed  subset $\tilde{B}$ of the mass space $\mathbb{C}^5$ such that the if $(m_i)=(m_1,...,m_5) \in \mathbb{C}^5\setminus \tilde{B}$ then $\pi_2^{-1}(m_i)$ is finite.
\end{theorem}
\begin{proof}
Since dimension of $\Omega$ is $5$, by the fiber dimension theorem, the fiber of the restrictions of the projection $\pi_2$ to the irreducible components of $\Omega$ are  finite or empty for a generic choice of mass $(m_i)\in \mathbb{C}^5$. This follows the result.
\end{proof}

\begin{theorem}
There is a proper open set $B$ of  $\mathbb{R}^5$ such that if $(m_i)=(m_1,...,m_5)\in \mathbb{R}^5\setminus B$ the number of CC6BP is finite. 
\end{theorem}
\begin{proof}
Define $B=\tilde{B}\cap \mathbb{R}^5$. Since $\tilde{B}$  is a proper closed subset $\mathbb{C}^5$, $B$ is a proper closed subset of $\mathbb{R}^5$. If $(m_i)\in \mathbb{R}^5\setminus B$ then $\pi_2^{-1}(m_i)$ is finite or empty. In both cases the number of possibilities for the mutual distances $r_{ij}$ associated to a CC6BP with masses $m_1,..,m_4,m_5=m_6$ is finite. Each vector of mutual distances  determines a unique CC6BP . This proves the result. 
\end{proof}

%


\begin{thebibliography}{HMMNWW}

\bibitem{AA} A. Albouy, { \it Symétrie des configurations centrales de quatre corps.} Comptes rendus de l'Académie des sciences. Série 1, Mathématique, 320.2 (1995), 217-220.

\bibitem{AFS} A. Albouy Y. Fu, and S. Sun. {\it Symmetry of planar four-body convex central configurations.} Proceedings of the Royal Society of London A: Mathematical, Physical and Engineering Sciences. Vol. 464. No. 2093. The Royal Society, (2008), 1355-1365

\bibitem{AK} A. Albouy and V.  Kaloshin, {\it Finiteness of central configurations of five bodies in the plane.} Annals of mathematics, (2012), 535-588.

\bibitem{FL} F. Ced\'o, and J. Libre. { \it Symmetric central configurations of the spatial n-body problem.} Journal of Geometry and Physics 6.3 (1989), 367-394.


\bibitem{CLO} D. Cox, J. Little and D. O'Shea, {\it Ideals, varieties, and algorithms.}  Springer, New York, 2007.


\bibitem{CZ}C. Deng, and S. Zhang. { \it Planar symmetric concave central configurations in Newtonian four-body problems.} Journal of Geometry and Physics 83, (2014), 43-52.

\bibitem{DGPS}  W. Decker, G.-M. Greuel, G. Pfister, and H. Sch{\"o}nemann,   {\sc Singular} {4-1-1} --- {A} computer algebra system for polynomial computations, 2018, \url{http://www.singular.uni-kl.de}.

\bibitem{DIG} { \sloppy T. Dias, {\it Finiteness of Cross Central Configurations, GitHub repository}, 2018, \url{https://github.com/thiagodiasoliveira/CC6BP}}

\bibitem{DI} T. Dias { \it New equations for central configurations and generic finiteness.} Proceedings of the American Mathematical Society 145.7 (2017), 3069-3084.




\bibitem{GL}M., Gidea, and J. Llibre. { \it Symmetric planar central configurations of five bodies: Euler plus two.} Celestial Mechanics and Dynamical Astronomy 106.1 (2010), 89.




\bibitem{HA} Y. Hagihara, {\it Celestial Mechanics I.} MIT Press, Cambridge, MA, 1970.

\bibitem{HM1} M. Hampton and R. Moeckel,  {\it Finiteness of relative equilibria of the four-body problem.} Inventiones mathematicae, 163(2), (2006), 289-312.

\bibitem{HM2} M. Hampton and R. Moeckel,  {\it Finiteness of stationary configurations of the four-vortex problem.} Trans. Amer. Math. Soc. 361(3), (2009), 1317-1332.

\bibitem{HRS} M. Hampton, G.E. Roberts, M. Santoprete, {\it Relative equilibria in the four-vortex problem with two pairs of equal vorticities.} J. Nonlinear Sci. 24(1), (2014), 39-92.

\bibitem{LE} E. S. G. Leandro, { \it Finiteness and bifurcations of some symmetrical classes of central configurations.} Archive for rational mechanics and analysis 167.2 (2003), 147-177. 

\bibitem{JS} L. Lei, and M. Santoprete. { \it Rosette central configurations, degenerate central configurations and bifurcations.} Celestial Mechanics and Dynamical Astronomy 94.3 (2006), 271-287.

\bibitem{LMS} J. Llibre, R. Moeckel, and C. Sim\'o, {\it Central configurations, periodic orbits, and Hamiltonian systems.} Birkhäuser, Basel, 2015.

\bibitem{RM1} Moeckel, R.  { \it Relative equilibria of the four-body problem.} Ergodic Theory and Dynamical Systems, 5(3), (1985), 417-435.

\bibitem{RM2} R. Moeckel  { \it Generic finiteness for Dziobek configurations.} Transactions of the American Mathematical Society 353.11 (2001), 4673-4686.

\bibitem{MO} J. Montaldi, { \it Existence of symmetric central configurations.} Celestial Mechanics and Dynamical Astronomy 122.4 (2015), 405-418.

\bibitem{ON}  K. A. O’Neil,  { \it Stationary configurations of point vortices.} Transactions of the American Mathematical Society, 302(2), (1987), 383-425.

\bibitem{PST} E. Perez-Chavela, M. Santoprete, C. Tamayo, { \it Symmetric relative equilibria in the four-vortex problem with three equal vorticities.} Dyn. Contin. Discrete Impuls. Syst. Series A Math. Anal. 22, (2015), 189-209.

\bibitem{sage}   SageMath, the Sage Mathematics Software System (Version 8.1),
   The Sage Developers, 2018, \url{http://www.sagemath.org}.
   
\bibitem{SH} I. R. Shafarevich, { \it Basic Algebraic Geometry 1, Varieties in Projective Space}, Springer-Verlag,
Berlin, Heidelberg, New York (1994)

\bibitem{SKKT} K. Smith, , L. Kahanpää, P. Kekäläinen, and W. Traves, { \it  An Invitation to Algebraic Geometry.} Springer Science \& Business Media. 2004.

\bibitem{SM} S. Smale, {\it Mathematical problems for the next century.} The mathematical intelligencer, 1998, 20(2), 7-15.


\bibitem{TS} Y.L. Tsai, {\it Numbers of Relative Equilibria in the Planar Four-Vortex Problem: Some Special Cases.} Journal of Nonlinear Science, 27(3),(2017) , 775-815.

\bibitem{WT}   H. Whitney. {\it Elementary structure of real algebraic varieties.} Annals of Math. 66(3), 1957, 545-556.

\bibitem{XZ} X. Yu, and S. Zhang. { \it Twisted angles for central configurations formed by two twisted regular polygons.} Journal of Differential Equations 253.7 (2012), 2106-2122.


\end{thebibliography}

\section{Application On Finiteness Of Cross Central Configurations Of The Six-Vortex Problem} \label{Vortexcase}
 
Consider $n$ point vortices with positions $x_i\in \mathbb{R}^2$ and vortex strengths $\gamma_i\neq 0$. The motion of the particles is described
by the \emph{Helmholtz's equations}
\begin{equation}\label{Helmeq}
\gamma_j\dot{x}_i=J\frac{\partial H}{\partial x_i}=-J\displaystyle \sum_{\begin{subarray}{c}j=1\\ j\neq i\end{subarray}}^{n}\frac{\gamma_i\gamma_j}{r_{ij}^2}(x_i-x_j), \quad 1\leq i<j\leq n,
\end{equation} 
for which the mutual distances are $r_{ij}=||x_i-x_j||\neq 0$, the \emph{vortex potential} is 
$H=-\sum_{i<j}\gamma_i\gamma_j\text{ln}(r_{ij})$, and $ J=\left(\begin{smallmatrix}
0&1\\
-1&0
\end{smallmatrix}\right).$

A configuration $x=(x_1,...,x_n)\in(\R^2)^n$ with vortex strengths $\gamma_i\neq 0$ is called a \emph{central configuration of $n$-vortex problem} if there exists $\lambda \neq 0$ such that
\begin{equation}\label{eqv}
\sum_{j\neq i}\gamma_{j}(x_{j}-x_{i})r_{ij}^{-2}+\lambda(x_i-c)=0, \qquad i=1,...,n,
\end{equation}
where the total vorticity $\gamma=\gamma_1+\cdots +\gamma_n \neq 0$ and the \emph{center of vorticity} is given by $c=\frac{1}{\gamma}\left(\gamma_1x_1+\cdots +\gamma_nx_n\right)$.

The solutions of the system \eqref{Helmeq} in which the bodies execute a rigid rotation with angular velocity $\lambda\neq 0$ has central configurations as initial conditions. These special solutions are called \emph{relative equilibria}. See \cite{HM2}  and \cite{ON} for more details about special solutions of the $n$-vortex problem.

In this section, we consider the central configurations of six-vortex with the same geometric conditions of the CC6BP described in section $\ref{formul}$. These configurations will be named \emph{cross configurations} of the  6-vortex problem. From now on, we denote these configurations by the term CC6VP.

The Laura-Andoyer equations for planar and the non-colinear central configurations of vortex with six bodies are given by 
\begin{equation}\label{LAV}
LV_{ij}=\sum_{k \neq i,j} \gamma_{k}v_{ikj}\Delta_{ijk}=0,
\end{equation}
where $v_{ikj}=r_{ik}^{-2}-r_{jk}^{-2}$.

Analogously to the Newtonian case,  a noncollinear planar vortex central configuration $x$ with $c=0$ satisfies  \eqref{LAV} if and only if  satisfies \eqref{eqv}. The following Lemma~{$\ref{tt}$} is similar to Lemma~{$3.1$}.
\begin{lemma}\label{tt}
 If $x$ is a CC6VP then $\gamma_5=\gamma_6$. 
\end{lemma}

%
Using the symmetry conditions of CC6VP and the equations \eqref{rDelta}  the Laura-Andoyer system \eqref{LAV} becomes: 
\begin{align}\label{4LAV}
LV_{15}&=\gamma_2v_{125}r_{12}+\gamma_3v_{135}r_{13}+\gamma_4v_{145}r_{14}+2\gamma_5v_{165}(r_{14}-1)=0;\\ \nonumber
LV_{25}&=-\gamma_1v_{215}r_{12}+\gamma_3v_{235}r_{23}+\gamma_4v_{245}r_{24}-2\gamma_5v_{265}(1-r_{24})=0;\\ \nonumber
LV_{35}&=-\gamma_1v_{315}r_{13}-\gamma_2v_{325}r_{23}+\gamma_4v_{345}r_{34}-2\gamma_5v_{365}(1-r_{34})=0;\\ \nonumber
LV_{45}&=-\gamma_1v_{415}r_{14}-\gamma_2v_{425}r_{24}-\gamma_3v_{435}r_{34}-2\gamma_5v_{465}=0. \nonumber
\end{align}

If we can find a particular example $x$ of CC6VP which make the rank of
$$AV(x)=\left(\begin{array}{rrrr}
0 & -v_{215} r_{12} & -v_{315} r_{13} & v_{415} r_{14} \\
v_{125} r_{12} & 0 & -v_{325} r_{23} & v_{425} r_{24} \\
v_{135} r_{13} & v_{235} r_{23} & 0 & v_{435} r_{34} \\
v_{145} r_{14} & v_{245} r_{24} & v_{345} r_{34} & 0
\end{array}\right).$$
  equals $4$, then we can use the same argument given in section \ref{FN} to prove the following theorem:
\begin{theorem}
There is a proper open set $C$ of the $\mathbb{R}^5$ such that if $(\gamma_1,...,\gamma_5)\in \mathbb{R}^5\setminus C$ the number of CC6VP is finite. 
\end{theorem}
%

\begin{prop}\label{importantexample V}
Suppose a configuration of points $x$  If $\gamma_1=\gamma_2=\gamma_3=\gamma_4=1$,  $\gamma_5=\gamma_6$, there is a unique a unique CC6VP $\overline{x}$ satisfying the three geometric conditions described on proposition \ref{importantexample}.
\end{prop} 
\begin{proof}
Firstly, we write each $r_{ij}$ only depending on $r_{23}$:
\begin{align}\label{r23}
&0<r_{23}=r_{25}<2;\\ \nonumber
&r_{12}=r_{34}=\frac{2-r_{23}}{2};\\ \nonumber
&r_{13}=r_{24}=\frac{2+r_{23}}{2};\\ \nonumber
&r_{15}=r_{16}=r_{45}=r_{46}=\sqrt{\frac{{r^2_{23}}+4}{4}};\\ \nonumber
&r_{25}=r_{26}=r_{35}=r_{36}=\frac{\sqrt{2}{r_{23}}}{2}. \nonumber
\end{align}
Using equations~\eqref{4LAV} and the relations $\eqref{r23}$, we get two polynomial $FV_1$ and $FV_2 \in \mathbb{C}[R_{23}][\Gamma_5]$ 
\begin{align}\label{lavr23}
FV_1&=-16 - 32 R^2_{23} + R^4_{23} + (16 - r^4_{23})\Gamma_5;\\ \nonumber
FV_2&=128 - 16 R^2_{23} - 40 R^4_{23} + R^6_{23} + (64  - 64 R^2_{23} + 12 R^4_{23}) \Gamma_5. \nonumber
\end{align}
Note that every pair $(r_{23},\gamma_5)$ provided by a CC6VP satisfying the relations \eqref{r23} belongs to $Z(FV_1,FV_2).$ 
The computation of the Resultant of $FV_1$ and $FV_2$ with respect to $\Gamma_5$ yields the polynomial
$$FV_3(R_{23})=768 + 384 R^2_{23} - 576 R^4_{23} - 24 R^6_{23} + R^8_{23}.$$

Since $0<r_{23}<2$, we only need to study the roots of $FV_3$ in the interval $[0,2]$. Making the change of variables $R_{23}^2=Y$, we can consider the function $p(Y)=768 + 384 Y - 576 Y^2 - 24 Y^3 + Y^4$, $Y\in[0,4]$. 
Because $p''(Y)<0$ for all $Y\in(0,4)$, $p(0)>0$ and $p(4)<0$, then $P(Y)$ has unique root in $[0,4]$.
Consequently, there has unique root $\overline{r_{23}}\in[0,2]$ of $FV_3$.


Therefore, $\overline{r}_{23}$ is partial solution of the system \eqref{lavr23} and $$1.217014>\overline{r}_{23}>1.217013.$$  The extension theorem \ref{LAV}, implies that there is exactly one $\overline{\gamma}_5\in\R^{+}$  satisfying the two equations of \eqref{lavr23} with $\overline{r}_{23}$.
\end{proof}

Now we only have to check the above example $\overline{x}$ of Proposition~\ref{importantexample V}  can make the rank of the matrix $AV(\overline{x})$ is $4$.
\begin{lemma}\label{AV}  The example  $\overline{x}$ of Proposition~\ref{importantexample V} implies $\text{rank}\left(AV\right)(\overline{x})=4$.
\end{lemma}
\begin{proof}
If $x$ is a {CC6VP} that satisfies the conditions of proposition \ref{importantexample V}, then the determinant of $AV(x)$ is $$\frac{159744 r^2_{23} - 6144 r^4_{23} - 28416 r^6_{23} - 58624 r^8_{23} - 880 r^{10}_23 + 
 104 r^{12}_{23} - r^{14}_{23}}{4 (-2 + r_{23})^4 r^4_{23} (2 + r_{23})^4 (4 + r^2_{23})^2}.$$
We do the factorization of the numerator of the determinant of $AV(x)$, then we have $$r^2_{23} (-192 - 144 r^2_{23} - 84 r^4_{23} + r^6_{23}) (832 - 656 r^2_{23} - 20 r^4_{23} + r^6_{23}).$$
Because $r^2_{23} (-192 - 144 r^2_{23} - 84 r^4_{23} + r^6_{23})<0$ for all $r_{23}$ in $(0,2)$, 
then we only consider the root of the function $FV_4$ with variable $R_{23}\in[0,2]$:  $$FV_4(R_{23})=832 - 656 R^2_{23} - 20 R^4_{23} + R^6_{23}.$$
It is sufficient to study the function $u(Y)=832 - 656Y - 20Y^2 + Y^3$ for $Y$ in $(0,4)$. Because $u'(Y)<0$, $u(0)>0$ and $u(3)<0$, then $u(Y)$ has only one root. Therefore, there is unique $r'_{23}\in(0,2)$ satisfies $FV_4(r'_{23})=0$ and $$1.106943>r'_{23}>1.106942.$$
Thus, the rank of the matrix $AV(\overline{x})$ is $4$, because $\overline{r_{23}}\neq{r'_{23}}$.
\end{proof}

\section*{Acknowledgments}

The authors wish to thank  Kuo-Chan Cheng for their hospitality and many stimulating conversations,  to   Alain Albouy for their continuous incentive, to Eduardo Leandro and Ya-Lun Tsai  for their helpful remarks, and to the Department of Mathematics at the National Tsing Hua University and the Department of Mathematics at the {\it Universidade Federal Rural de Pernambuco} for their support.
The first author was partly supported by the Ministry of Science and Technology of the Republic of China under the grant MOST 107-2811-M-007-004.

\end{document}